\documentclass[11pt,oneside]{amsart}

\usepackage{amsmath}
\usepackage{geometry}
\usepackage{hyperref}

\theoremstyle{theorem}
\newtheorem{theorem}{Theorem}
\newtheorem{MainTheorem}{Theorem}

\newtheorem{lemma}{Lemma}

\newcounter{abc}
\newtheorem{thmPrime}[abc]{Theorem}

\theoremstyle{definition}

\theoremstyle{remark}

\title[Rigidity of convex hypersurfaces]{Rigidity of convex hypersurfaces in multidimensional spaces of constant curvature}
\author{Alexander A. Borisenko}
\thanks{\textbf{Acknowledgements.} The author would like to thank Kostiantyn Drach for the help in preparation of the manuscript and for several useful remarks.}

\date{}

\address{B. Verkin Institute for Low Temperature Physics and Engineering of the National, 47 Nauky Ave., Kharkiv, 61103, Ukraine}
\address{Brown University -- ICERM, 121 South Main Street, Box E 11th Floor, Providence, RI 02903, USA}
\email{aborisenk@gmail.com}

\begin{document}

\begin{abstract}
In 1972, E.~P.~Senkin generalized the celebrated theorem of A.~V.~Pogorelov on unique determination of compact convex surfaces by their intrinsic metrics in the Euclidean 3-space $E^3$ to higher dimensional Euclidean spaces $E^{n+1}$ under a mild assumption on the smoothness of the hypersurface. In this paper, we remove this assumption and thus establish this rigidity result for arbitrary compact, convex hypersurfaces in $E^{n+1}$, $n \ge 3$. We also prove the corresponding results in other model spaces of constant curvature.  

\medskip	

\noindent
\textbf{Keywords: } rigidity; convex hypersurface; space of constant curvature.

\medskip
\noindent
\textbf{MSC2020: } 52A10, 52A55, 51M10, 53C22.
\end{abstract}

\maketitle

In 1950, A.\,V.~Pogorelov proved the following rigidity result for compact convex surfaces in Euclidean space $E^3$. 

\begin{MainTheorem}[\cite{Pog}]
\label{Thm:Pog}
Let $F_1$ and $F_2$ be a pair of compact, convex surfaces in $E^3$ isometric with respect to their intrinsic metrics. Then there exists an isometry of the ambient Euclidean space $E^3$ that maps the surface $F_1$ onto the surface $F_2$.
\end{MainTheorem}

There are no regularity assumptions on the surfaces in the theorem above. Only the convexity of surfaces must be assumed. Under stronger assumptions on regularity of surfaces, Theorem~\ref{Thm:Pog} was proven by S.~Cohn-Vossen in 1924 \cite{CV} and G.~Herglotz in 1943 \cite{Her}. A.\,V.~Pogorelov generalized Theorem~\ref{Thm:Pog} for general convex surfaces in the spherical space $\mathbb S^3$. Using Pogorelov's, A.\,D.~Alexandrov's, and E.\,P.~Senkin's results, A.\,D.~Milka proved the result analogous to Theorem~\ref{Thm:Pog} in the hyperbolic (Lobachevsky) space $\mathbb H^3$. E.\,P.~Senkin generalized Pogorelov's theorem for Euclidean spaces of arbitrary dimension \cite{Sen} but with additional assumptions on regularity of hypersurfaces.

\begin{theorem}[\cite{Sen}]
\label{Thm:Sen}
Let $F_1, F_2$ be a pair of compact, convex, $C^1$-smooth hypersurfaces in a multidimensional Euclidean space $E^{n+1}$. If $F_1$ and $F_2$ are isometric with respect to their intrinsic metrics, then there exists an isometry of the ambient space $E^{n+1}$ that maps one hypersurface onto the other.
\end{theorem}

In this paper, we will prove Theorem~\ref{Thm:Sen} without the assumption on regularity of hypersurfaces. More precisely, our goal is to establish the following result: 

\begin{thmPrime}
\label{Thm:Main}
Let $F_1$ and $F_2$ be a pair of compact, convex hypersurfaces in Euclidean space $E^{n+1}$, $n \ge 3$. If $F_1$ and $F_2$ are isometric with respect to their intrinsic metrics, then there exists a motion of $E^{n+1}$ that maps $F_1$ onto $F_2$.
\end{thmPrime}

This theorem is proven in a sequence of steps based on the following lemmas. We say that a hypersurface $F \subset E^{n+1}$ is \emph{visible from a point $Q \in E^{n+1} \setminus F$} if for every point $P \in F$ the ray $QP$ intersect $F$ only at $P$. We will further say that a point $P$ is visible from \emph{inside} if the ray $QP$ makes the acute angle with the outer normal to the supporting hyperplane to $F$ at $P$.  
 
We will also say that a pair of hypersurfaces is congruent if there exists a motion of $E^{n+1}$ that maps one hypersurface to the other.

\begin{lemma}[\cite{Sen}]
\label{Lem1}
Let $F_1$ and $F_2$ be a pair of isometric convex hypersurfaces in $E^{n+1}$. Suppose that they are visible from the points $Q_1$ and $Q_2$, respectively. Let $L_1$ and $L_2$ be the boundaries of $F_1$, respectively $F_2$ (if the hypersurfaces are compact, then the boundaries are the points $X_1 \in F_1$ and $X_2 \in F_2$ that correspond each other under the isometry). Assume that there exist hyperplanes $P_1$ and $P_2$ passing through $Q_1$, respectively $Q_2$, such that for each $i \in \{1,2\}$, the hypersurface $F_i$ lies in one half-space with respect to the hyperplane $P_i$. If the distances from the points $Q_1$, $Q_2$ to the corresponding under the isometry points of the boundaries $L_1$, respectively $L_2$, are equal, then either the hypersurfaces $F_1$ and $F_2$ are congruent, or there exists a motion $\phi$ of $E^{n+1}$ such that:
\begin{enumerate}
\item
$\phi(X_1) = X_2$ for some pair of points $X_1 \in F_1$ and $X_2 \in F_2$ that correspond each other under the isometry of the hypersurfaces; we keep the notation $F_1$ for $\phi(F_1)$;
\item
there exits a point $Q \in E^{n+1}$ and a pair of neighborhoods $U_i$ of $X_i$ in $F_i$ such that these neighborhoods are visible from $Q$ from inside; let $r_i$ denote the distance function from $Q$ to the points in $U_i$
\item
for every corresponding under the isometry points $X \in U_1$ and $X \in U_2$, we have
\[
r_1(X) < r_2(X).  
\]\qed
\end{enumerate}
\end{lemma}

For a general (not necessarily smooth) surface $F \subset E^3$, we say that $F$ has \emph{non-positive curvature} if for every point  on $F$ there exists its neighborhood that is impossible to cut a cup.

\begin{lemma}[\cite{Pog}, Ch.~IV, \S 2, p.213]
\label{Lem2}
Let $F$ be a $2$-dimensional convex surface in $E^3$ given in an explicit form
\[
z = z(x,y),
\]
where $x,y,z$ are some orthogonal Cartesian coordinates in $E^3$. Denote by $\xi(x,y)$ the $z$-coordinate of the infinitesimal bending field of the surface $F$, and define the surface $\Phi$ given by the equation
\[
z = \xi(x,y).
\] 

If $\Phi$ does not contain flat regions, then it has non-positive curvature everywhere. If $\Phi$ contains flat regions, then the curvature of $\Phi$ is non-positive everywhere except those flat regions. \qed
\end{lemma}

Let $F$ be the hypersurface given by the radius vector
\begin{equation}
\label{Eq:RadiusVector}
R = \frac{1}{2} (r_1 + r_2),
\end{equation}
where $r_1$ and $r_2$ are the radius vectors of $F_1$ and $F_2$ as in Lemma~\ref{Lem1}. By that lemma, for every $X$, $r_1(X) = P_1 \in F_1$ and $r_2(X) =: P_2 \in F_2$ are the pair of the corresponding under the isometry points of $F_1$, $F_2$, and $r_1(X_0) = r_2(X_0) = P_0$ for some point $P_0$ that satisfies Lemma~\ref{Lem1}. 

Under the additional assumption that the hypersurfaces $F_1$ and $F_2$ are $C^1$-smooth, the following lemma was proven:

\begin{lemma}
\label{Lem3}
The hypersurface $F$ with radius vector \eqref{Eq:RadiusVector} is a convex hypersurface in the neighborhood of the point $P_0$. For this hypersurface, the vector field $\sigma := r_1 - r_2$ is an infinitesimal bending field on the hypersurface $F$. It is Lipshitz and satisfies the equation
\[
\left<dR, d\sigma \right> = 0 \quad \text{ a.e. in the neighborhood of $P_0$.} 
\]
\end{lemma}

Let us define 
\[
E^3 := \text{span}(e_1, e_2, n),
\]
where $e_1, e_2$ are tangent vectors to $F$ at $P_0$, and $n$ is the normal vector at this point. The intersection $F \cap E^3 =: F^2$ is a compact convex surface in $E^3$. We will now work in the subspace $E^3$. In the neighborhood of the point $P_0$ the surface $F^2$ is given in the explicit form $z = z(x,y)$ and $z = \xi(x,y)$ is the $z$-coordinate of the infinitesimal bending field along the surface $F^2$. At $P_0$, the function $z = \xi(x,y)$ assumes its minimum. The plane $z = \varepsilon$, $\varepsilon > 0$, cuts from the surface $z = \xi(x,y)$ a cap for small $|\varepsilon|$. It contradicts Pogorelov's Lemma~\ref{Lem2}. Therefore, $r_1 = r_2$ and the hypersurfaces $F_1$ and $F_2$ coincide.

Now we prove Lemma~\ref{Lem3} without additional assumption of $C^1$ regularity of the hypersurfaces $F_1$ and $F_2$. We will require only that $F_1$ and $F_2$ are compact, isometric, convex hypersurfaces.

\section{Convex combination of isometric hypersurfaces}

In this section, we discuss some facts about convex combinations of convex hypersurfaces in $E^4$. 

At every point of a convex hypersurface in $E^4$ there exists the tangent cone. Such a cone is a convex hypersurface as well. Let $V^n$ be a strongly convex cone in the Euclidean space $E^{n+1}$, where a convex cone is called \emph{strongly convex} if at the vertex $O$ of the cone there exists a supporting hyperplane that intersects the cone only at $O$.

It is well-known that tangent cones $V^3$ at points of a convex hypersurface $F^3 \subset E^4$ have one of the following forms:
\begin{enumerate}
\item
\label{It:Form1}
$V^3$ is a strongly convex cone in $E^4$;
\item
\label{It:Form2}
$V^3 = V^2 \times E^1$ is a metric product of a strongly convex cone $V^2$ in $E^3$ and a Euclidean line $E^1$;
\item
\label{It:Form3}
$V^3 = V^1 \times E^2$ is a metric product of a strongly convex cone $V^1$ in $E^2$ and a Euclidean plane $E^2$;
\item
\label{It:Form4}
$V^3 = E^3$ is a Euclidean space $E^3$.
\end{enumerate}

If $P_1 \in F_1$, $P_2 \in F_2$ are the corresponding points in the convex isometric hypersurfaces $F_1$ and $F_2$, then the tangent cones these points are isometric too.

\begin{lemma}
\label{Lem4}
Let $F_1$ and $F_2$ be a pair of convex isometric hypersurfaces in $E^4$. 

I. Suppose that the tangent cone $K(P_1)$ at a point $P_1 \in F_1$ has the form \eqref{It:Form1}. Then for the corresponding under the isometry point $P_2 \in K_2$ the tangent cone $K(P_2)$ also has the form \eqref{It:Form1} and the cones $K(P_1)$, $K(P_2)$ are congruent.

II. If the cone $K(P_1)$ has the form \eqref{It:Form2}, i.e., $K(P_1) = V_1^2 \times E^1_1$, then $K(P_2)$ has the same from $K(P_2) = V_2^2 \times E^1_2$ and the cones $V_1^2, V_2^2$ are congruent. The edges $E^1_1, E^1_2$ correspond under the isometry of $K(P_1)$ and $K(P_2)$.
\end{lemma}

\begin{proof}
\noindent
I. Suppose $K(P_2)$ has one of the forms \eqref{It:Form2}, \eqref{It:Form3}, \eqref{It:Form4}. In each of these cases we can choose a straight segment $\gamma_2 \subset K(P_2)$ such that $P_2$ lies in the interior of $\gamma_2$. Since $K(P_1)$ and $K(P_2)$ are isometric, for $K(P_1)$ there exists a corresponding shortest line $\gamma_1 \subset K(P_1)$ through $P_1$; the curve $\gamma_1$ is isometric to $\gamma_2$. The point $P_1$ breaks $\gamma_1$ into two straight segments $\gamma_1^+$ and $\gamma_1^-$ with $P_1$ being their common boundary point.    

Let $E^3 = \text{span}(\gamma_1^+, \gamma_1^-, \ell)$, where $\ell$ is a ray inside the cone $K(P_1)$ and does not belong to the plane $\text{span}(\gamma_1^+, \gamma_1^-)$. The intersection $K(P_1) \cap E^3$ is a strongly convex cone in $E^3$; for this cone, $\gamma_1$ is the shortest line in this cone, this line passes through $P_1$ and this point lies in the interior of $\gamma_1$. This is a contradiction with the fact that on a strongly convex cone in $E^3$ a shortest line cannot go through the vertex of the cone.

Let us now show that $K(P_1)$ and $K(P_2)$ are congruent, i.e., there exists a motion of the Euclidean space $E^4$ that maps one cone onto the other. Let $S_i^2$, $i \in \{1,2\}$, be the unit spheres with the centers at the points $P_1, P_2$. Then $\tilde F_i^2 = K(P_i) \cap S_i^3$, $i \in \{1,2\}$, are compact convex isometric surfaces in open hemispheres of $S_1^3, S_2^3$. By moving the spheres, we can assume that $\tilde F_1^2$ and $\tilde F_2$ belong to the same spherical space. For them, we can apply the following theorem due to Pogorelov

\begin{MainTheorem}[\cite{Pog}]
\label{Thm:Pog2}
Compact isometric convex surfaces in the spherical space $S^3$ are congruent. 
\end{MainTheorem}
This finishes the proof of Part I of Lemma~\ref{Lem4}.

\medskip
\noindent
II. The proof of Part II is similar to the proof of Part I.
\end{proof}

\begin{lemma}
\label{Lem5}
Let $F_1$ and $F_2$ be a pair of convex isometric hypersurfaces in $E^4$. Suppose that at the point $P_1 \in F_1$ the tangent cone has the form either \eqref{It:Form3} or \eqref{It:Form4} from above. Then the tangent cone $K(P_2)$ at the corresponding under the isometry point $P_2 \in F_2$ has the form \eqref{It:Form3} or \eqref{It:Form4}. The following 3 possibilities can occur:
\begin{enumerate}
\item[a)]
both cones are dihedral angles $K(P_1)=V_1^1 \times E_1^2$, $K(P_2) = V_2^1 \times E_2^2$;
\item[b)]
one tangent cone is a hyperplane, and the other is a dihedral angle;
\item[c)]
both tangent cones are hyperplanes.
\end{enumerate}
\end{lemma}

Let $G_1$, $G_2$ be small neighborhoods of the points $P_1 \in F_1$ and $P_2 \in F_2$, $P_1 = P_2 = P_0$, which satisfy the assumptions of Lemma~\ref{Lem1}. Consider the cones $K(P_1)$ and $K(P_2)$. The following cases can occur:

\begin{enumerate}
\item[I)]
$K(P_1) = V^3$. Then the cones $K(P_1)$ and $K(P_2)$ coincide.
\item[II)]
$K(P_1) = V_1^2 \times E^1_1$. Then the cones $K(P_1)$ and $K(P_2)$ coincide too. By Lemma~\ref{Lem1}, $V_1^2 \subseteq V_2^2$. By isometry of $V_1^2$ and $V_2^2$, we obtain that $V_1^2 = V_1^2$ and the lines $E_1^1$ and $E_2^1$ coincide.
\item[III)]
\begin{enumerate}
\item[a)]
If both tangent cones are dihedral angles, then from Lemma~\ref{Lem1} it follows that the edges $E_1^2, E_2^2$ are corresponding under the isometry and coincide, and one dihedral angle lies inside the other.
\item[b)]
If both tangent cones are hyperplanes, then they coincide.
\item[c)]
If one cone is a hyperplane and the other cone is a dihedral angle, then the argument is similar to case a).
\end{enumerate}
In all cases a)-c) above, the linear combination of the cones at the point $P_0$ is a convex dihedral angle. 
\end{enumerate}
 
Let us treat the cases separately.

\medskip
\noindent
\textbf{I.} $K(P_1) = K(P_2) = V^3$. 

\smallskip
\textbf{I1.} Let $(X_1^n) \subset F_1$ and $(X_2^n) \subset F_2$ be a pair of sequences of corresponding under the isometry cone points, such that $X_1^n \to P_0$ and $X_2^n \to P_0$ as $n \to \infty$. Denote the limiting cones of the cones $K_1(X_1^n)$ and $K_2(X_2^n)$ as $K_1^0$ and $K_2^0$ respectively. By construction, $K_1^0$, $K_2^0$ are isometric supporting cones at $P_0$ to $F_1$ and $F_2$ respectively.

By Lemma~\ref{Lem4}, for each $n$, 
\[
K_1(X_1^n) = A_n K_2(X_2^n) + a_n,
\] 
where $a_n$ is a vector and $A_n$ is an orthogonal matrix. As $n \to \infty$, $a_n \to 0$ and $A_n \to A_0$, where $A_0$ is some orthogonal matrix. Since $K_1^0 = K_2^0$, we obtain $K_1^0 = A_0 K_2^0$, and thus $A_0 = I$ is the identity matrix. For large $n$, the matrices $I + A_n$ are non-degenerate and the convex combination of cones $K_1(X_1^n)$ and $K_2(X_2^n)$ is the cone $K(X^n) = (I + A_n) \cdot K_2(X_2^n) + a_n$. We obtain that $K(X^n)$ is a non-degenerate affine image of $K(X_2^n)$, and hence is convex.

\smallskip
\textbf{I2.} Let $K_1(X_1^n) = V_1^2(n) + E_1^1(n)$, $K_2(X_2^n) = V_2^2(n) + E_2^1(n)$. If $K_1^0 = V_1^2 \times E_1^1$ and $K_2^0 = V_2^2 \times E_2^1$, then the isometric directions $\ell_1^0 \in V_1^2$, $\ell_2^0 \in V_2^2$ belong to the tangent cones $K(P_1) = K(P_2)$. It follows that $K_1^0 = K_2^0$, $V_1^2 = V_2^2$, $E_1^1 = E_2^1$. The curvature of $V_1^2$ is greater than some $\alpha_0 > 0$. We obtained that the angles between any pair of isometric directions in the cones $V_1^2(n)$, $V_2^2(n)$ is less than $\epsilon(n)$, where $\epsilon(n) \to 0$ as $n \to \infty$, the curvature at the vertices are at least $\theta_0 > 0$, and the ball $\omega$ belongs to the both cones. We will now show that for sufficiently big $n$ the convex combination of cones $K_1(X_1^n)$ and $K_2(X_2^n)$ is again a convex cone. It is enough to prove that the cone
\[
K(X^n) = K_1(X_1^n) + K_2(X_2^n)
\]  
is locally convex. For this we need to show that through every $2$-dimensional generator $t_0$ of $K(X^n)$ it is possible to draw a hyperplane such that all generators close to $t_0$ lies in the halfspace that contains the ball $\omega$. Assume the contrary, i.e., for each $n$ there exists a generator $t_0^n$ that does not satisfy the locally convex condition. Let $t_1^n \in K_1(X_1^n)$ be the corresponding generator of $K_1(X_1^n)$. The sequence of generators $t_1^n$ converges to the generator $t_1^0$ of the convex cone $K_1^0$. Each generator $t_1^n$ is a metric product of generators $\ell_1^n \in V_1^2(n)$ and $E_1^1(n)$. Let $A_1^n$ be the point on $\ell_1^n$ at distance $1$ from the vertex of the cone $V_1^2(n)$, and let $\mathcal D_1^n$ be the tangent dihedral angle at the point $A_1^n$ for the cone $K_1(X_1^n)$. We define the same objects $t_2^n$, $\ell_2^n$, $A_2^n$, $\mathcal D_2^n$ for the cone $K_2(X_2^n)$. The $2$-dimensional edges of $\mathcal D_1^n$ and $\mathcal D_2^n$ are corresponding under the isometry. The convex combination of $\mathcal D_1^n$ and $\mathcal D_2^n$ for sufficiently big $n$ is a dihedral angle $\mathcal D^n$. By construction, the ball $\omega$ is inside $\mathcal D^n$. There exists a supporting hyperplane to $\mathcal D^n$, which we call $\Pi$, that passes through the edge of $\mathcal D^n$. Now, we follow Pogorelov's proof \cite[Lemma 1, pp.137-136]{Pog}. 

Let $\bar n$ be the normal to $\Pi$. When moved to a point $A_i^n$, $i \in \{1,2\}$, the vector $\bar n$ points inside the cone $K_i(X_i^n)$. Connect the point $A_1^n$ with the shortest line $\gamma_1^n$ to a point $B_1^n \in V_1^2(n)$ near $A_1^n$. Let $r_1(s)$ be the radius-vector of $\gamma_1^n$, where $s$ is the arc-length parameter on $\gamma_1^n$ chosen so that $s=0$ corresponds to the point $A_1^n$. And let $r_2(s)$ be the radius vector of the corresponding under the isometry shortest line $\gamma_2^n \subset V_2^2(n)$. If $s=0$, 
\begin{equation}
\label{Eq:22}
\frac{d}{ds} \left<r_1 + r_2, \bar n\right> \ge 0.
\end{equation}

By Liberman's theorem \cite[p.~58]{Pog}, the inequality \ref{Eq:22} is true for all $s$ along the shortest line $\gamma_1^n$. Integrating this inequality, we obtain that all points of the cone $K(X^n)$ close to the image the point $A_1^n$ lie from one side with respect to the supporting hyperplane with the inner normal $\bar n$. This implies that the cone $K(X^n)$ is locally convex.

We take a small neighborhood of the point $P_0 = P_1 = P_2$ in the hypersurfaces $F_1$ and $F_2$. Let $F$ be the convex combination of $F_1$ and $F_2$. The radius-vector of $F$ is $r = (r_1 + r_2)/2$, where $r_i$ is the radius vector of $F_i$. From above it follows that there exists a neighbourhood of the point $P_0 \in F$ such that $F$ is a convex hypersurface. The vector field $\sigma = r_1 - r_2$ is an infinitesimal bending vector field of $F$, i.e., $\left<dr, d\sigma\right> = 0$. After this we follow Senkin's original proof. This proves the uniqueness theorem for compact convex hypersurfaces in the Euclidean space $E^4$.  

\smallskip
\textbf{I3.} Let 
\begin{enumerate}
\item[a)] 
$K_1(X_1^n) = V_1^1(n) \times E_1^2(n)$, $K_2(X_2^n) = V_2^1(n) \times E_1^2(n)$.
\item[b)]
$K_1(X_1^n) = V_1^1(n) \times E_1^2(n)$, $K_2(X_2^n) = E^3(n)$.
\item[c)]
$K_1(X_1^n) = E_1^3(n)$, $K_2(X_2^n) = E_2^3(n)$.
\end{enumerate}
At these, the dihedral angles or support hyperplanes are equal. The convex combination of the cones $K_1(X_1^n)$, $K_2(X_2^n)$ are, correspondingly,
\begin{enumerate}
\item[a)]
a hyperplane;
\item[b)]
a dihedral angle;
\item[c)]
a convex cone.
\end{enumerate}

We proved that there exist neighborhoods of the point $P_0$ in $F_1$ and $F_2$ such that the convex combination of $F_1$ and $F_2$ within those neighborhoods is a convex hypersurface $F$.   

\medskip
\noindent
\textbf{II.} $K_1(P_1) = K_2(P_2) = V_2 \times E^1$.

\medskip
\noindent
\textbf{III.} a) If $K_1(P_1) = E^3$, $K_2(P_2) = E^3$, then $K_1^0$ and $K_2^0$ coincide the tangent hyperplane. 

b) If $K_1(P_1) = V_1^1 \times E_1^2$, $K_2(P_2) = V_2^1 \times E_2^2$, then the edges $E_1^2$, $E_2^2$ coincide and correspond under the isometry of the cones, and one dihedral angle lies inside another. In this case, the cones $K_1^0$ and $K_2^0$ are contained inside the cones $K_1(P_1)$ and $K_2(P_2)$. Similar to case \textbf{I}, we can prove that there exist neighborhoods of the points $P_1 \in F_1$ and $P_2 \in F_2$ such that the convex combination of $F_1$ and $F_2$ is the convex surface $F$.

\section{Proof of Theorem~\ref{Thm:Main}}

In this section we will prove the uniqueness of compact convex isometric hypersurfaces in $E^{4}$ without assumption on regularity (Theorem~\ref{Thm:Main}).

We will need a concept of \emph{Pogorelov map} \cite{Pog}. Let $F_1$ and $F_2$ be compact isometric convex hypersurfaces in the open hemisphere of the spherical space $S^n \subset E^{n+1}$. Let $x^0, x^1, \ldots, x^n$ be the Cartesian orthogonal coordinates in $E^{n+1}$ and $S^n$ is a sphere centered at the origin. We assume that $F_1$ and $F_2$ have the same orientation and belong to the hemisphere $x^0 > 0$. Let $r_1, r_2$ be the radius vectors of $F_1, F_2$ with the same parameters at the isometric points. Finally, let $\Phi_1$, $\Phi_2$ be the hypersurfaces in $E^{n}$ defined by the radius-vectors
\[
R_1 := \frac{r_1 - e_0 \left<r_1, e_0\right>}{\left<e_0, r_1 + r_2\right>}, \quad R_2 := \frac{r_2 - e_0 \left<r_2, e_0\right>}{\left<e_0, r_1 + r_2\right>},
\]
where $e_0$ is the unit coordinate vector corresponding to $x^0$. For $n = 4$, A.~V.~Pogorelov proved that $\Phi_1, \Phi_2$ are compact convex isometric hypersurfaces in $E^4$. This result is true for any $n$ and the proof is similar to that of Pogorelov's for $n=4$. For $n=4$, the uniqueness theorem in $S^4$ follows from the uniqueness theorem in $E^4$.

Now we prove Theorem~\ref{Thm:Main} in the hyperbolic space $\mathbb H^n$, $n=4$. In 1980, A.~D.~Milka proved Theorem~\ref{Thm:Main} in $\mathbb H^3$ \cite{Mil}. He used E.~P.~Senkin's idea of the proof of Theorem~\ref{Thm:Sen}. It is possible to move the surfaces $F_1$ and $F_2$ in such a way that they satisfy Lemma~\ref{Lem1}. This means that from some point $O$ one can see $F_1$ and $F_2$ from the same side and do Carmo and Warner proved uniqueness of compact regular convex hypersurfaces in $S^n$ \cite{dCW}. Gorsij generalized this theorem to compact convex isometric $C^1$-smooth hypersurfaces in $S^n$ \cite{Gor}. It was also proven that the images of convex isometric hypersurfaces in $S^n$ under the Pogorelov map are convex hypersurfaces in the Euclidean space $E^n$ is the hypersurfaces in the sphere can be seen from the convexity side \cite{Pog}. Milka proved an analogous result for convex isometric hypersurfaces in the hyperbolic space under the condition that it is possible to see convex isometric hypersurfaces from different convexity sides.

\begin{proof}[Proof of Theorem~\ref{Thm:Main}]
The proof is by induction on the dimension $n$. Suppose we have proven the result for $E^n, S^n, \mathbb H^n$. Let us show how to prove it for $E^{n+1}, S^{n+1}, \mathbb H^{n+1}$. 

For the convex hypersurfaces $F_s^n \subset E^{n+1}$, $s \in \{1,2\}$, the tangent convex cones are of the form
\[
K = V^{n-i} \times E^i, \quad i \in \{0, \ldots, n\}, 
\]
where $V^{n-i}$ is a strongly convex cone in $E^{n-i+1}$.

1) Let $K_1 = V_1^{n-i} \times E_1^i$ ($i \le n-3$) be a convex cone in $E^{n+1}$, $K_2$ be an isometric convex cone in $E^{n+1}$. Then $K_2 = V_2^{n-i} \times E_2^i$ is congruent to $K_1$. The proof follows the same way as we proved Lemma~\ref{Lem4} and we use the uniqueness of compact isometric convex hypersurfaces in $E^{n-1}$, $S^{n-1}$.

2) Let $K_1 = V_1^2 \times E_1^{n-2}$, and $K_2$ be an isometric convex cone in $E^{n+1}$. Then $K_2 = V_2^2 \times E_2^{n-2}$. The cones $K_1, K_2 \subset E^3$ are convex isometric cones.

3) Let $K_1 = V_1^1 \times E_1^{n-1}$, then $K_2 = V_2^1 \times E_2^{n-1}$ or $K_2 = E^n$. 

We prove Theorem~\ref{Thm:Main} similarly to the proof for $E^4, S^4, \mathbb H^4$ by induction.  
\end{proof}

\end{document}